\documentclass[11pt]{article}
\usepackage{mathrsfs}
\usepackage[centertags]{amsmath}
\usepackage{amsfonts}
\usepackage{amssymb}
\usepackage{amsthm}
\usepackage{cases}
\usepackage{indentfirst}
\usepackage{epsfig}
\usepackage{graphicx}
\usepackage{color}
\usepackage{multirow}
\usepackage{booktabs}
\usepackage{threeparttable}
\usepackage{graphicx}
\usepackage{subfigure}
\usepackage{placeins}
\allowdisplaybreaks[4]
\usepackage{setspace}
\numberwithin{equation}{section}

\definecolor{c20}{rgb}{0.,0.7,0.}
\definecolor{c30}{rgb}{0.,0.,1.}
\definecolor{c40}{rgb}{1,0.1,0.7}
\definecolor{c50}{rgb}{1,0,0}

\def\GED{\operatorname*{GED}}

\def\P{\operatorname*{P}}

\newtheorem{theorem}{Theorem}[section]

\newtheorem{lemma}{Lemma}[section]

\newtheorem{remark}{Remark}[section]

\numberwithin{equation}{section}

\begin{document}
\title{Distributional expansions of powered order statistics \\ from general error distribution}
\author{Yingyin Lu\qquad  Zuoxiang Peng\\
{School of Mathematics and Statistics, Southwest University, Chongqing, 400715, China}}
\date{}
\maketitle

\begin{quote}
{\bf Abstract.} Let $\{X_{n}, n\ge 1\}$ be a sequence of independent random variables with common general error distribution $\GED(v)$ with shape parameter $v>0$, and let $M_{n,r}$ denote the $r$th largest order statistics of $X_{1}, X_{2}, \cdots, X_{n}$. With different normalizing constants the  distributional expansions of normalized powered order statistics $|M_{n,r}|^{p}$ are established, from which the convergence rates of powered order statistics to their limits  are derived. This paper generalized Hall's results on powered-extremes of normal sequence.



\end{quote}

\section{Introduction}

The general error distribution is an extension of normal distribution. The density of general error distribution, say $g_{v}(x)$, is given by
\begin{eqnarray}
\label{eq1.1}
g_v(x)=\frac{v\exp(-\frac{1}{2}|\frac{x}{\lambda}|^v)}{\lambda 2^{1+\frac{1}{v}}\Gamma(\frac{1}{v})}, \quad x\in \mathbb{R},
\end{eqnarray}
where $v>0$ is the shape parameter, $\lambda=[2^{-2/v}\Gamma(1/v)/\Gamma(3/v)]^{1/2}$ with $\Gamma(\cdot)$ being the Gamma function. It is known that $g_{2}(x)=\frac{1}{\sqrt{2\pi}}e^{-x^{2}/2}$, is the standard normal density; and $g_{1}(x)=\frac{1}{\sqrt{2}}e^{-\sqrt{2}|x|}$ is the density of Laplace distribution.

For the general error distribution with parameter $v$ (denoted by $\GED(v)$), Peng et al. (2009) considered its tail behavior, and Peng et al. (2010) established the uniform convergence rates of normalized extremes under optimal normalizing constants, and Jia and Li (2014) established the higher-order distributional expansions of normalized maximum. Peng et al. (2010) and Jia and Li (2014) showed that the optimal convergence rate is proportional to $1/\log n$, similar to the results of Hall (1979) and Nair (1981) on normal extremes. In order to improve the convergence rate of normal extremes, Hall (1980) studied the asymptotics of $|M_{n,r}|^{p}$, the powered $r$th largest order statistics,  and showed that the distribution of normalized $M_{n,r}^{2}$ converges to its limit at the rate of $1/(\log n)^{2}$ under optimal normalizing constants, while the convergence rates are still the order of $1/\log n$ in the case of $p\ne 2$. For more details, see Hall (1980).  For more work on higher-order expansions of powered-extremes from normal samples, see Zhou and Ling (2016) for distributional expansions and Li and Peng (2018) for moment expansions. For other related work on distributional expansions of extremes, see Liao and Peng (2012) for lognormal distribution, Liao et al. (2014a, 2014b) for logarithmic general error distribution and skew-normal distribution, and Hashorva et al. (2016), Liao and Peng (2014, 2015), Liao et al. (2016) and Lu and Peng (2017) for bivariate H\"{u}sler-Reiss models.

Motivated by the work of Peng et al. (2010), Jia and Li (2014) and Hall (1980), the objective of this paper is to consider the higher-order expansions of powered order statistics from $\GED(v)$ sample. Let $X_{1}, X_{2}, \cdots, X_{n}$ be independent identically distributed random variables with common distribution function $G_{v}$ following $\GED(v)$ with density given by \eqref{eq1.1}. For positive integer $r$, let $M_{n,r}$ denote the $r$th largest order statistics of $X_{1}, X_{2}, \cdots, X_{n}$ and $M_{n}=M_{n,1}=\max_{1\le k\le n}\{X_{k}\}$ for later use. It is known that the distributional convergence rate of normalized maximum may depend heavily on the normalizing constants, see Leadbetter et al. (1983), Resnick (1987), Hall (1979, 1980) and Nair (1981) for normal samples. For the $\GED(v)$ distribution and positive power index $p$, it is necessary to discuss how to find the normalizing constants $c_{n}>0$ and $d_{n}$ such that
\begin{equation}\label{add0}
\lim_{n\to\infty}\Big(\P\{|M_{n,r}|^{p}\le c_{n}x+d_{n}\}- \Lambda_{r}(x)\Big)=0,
\end{equation}
where $\Lambda_{r}(x)=\Lambda(x)\sum_{j=0}^{r-1}e^{-jx}/j!$ with $\Lambda(x)=\exp(-\exp(-x)), x\in \mathbb{R}$, and further work on distributional expansions and convergence rates of $|M_{n,r}|^{p}$ with different normalizing constants.

For $v > 0$ with normalizing constants $\alpha_{n}$ and $\beta_{n}$ given by
\begin{eqnarray}\label{eq1.2}
\begin{cases}
\alpha_{n}=\frac{2^{1/v}\lambda}{v(\log n)^{1-1/v}};\\
\beta_{n}=2^{1/v}\lambda \left(\log n\right)^{1/v} -
\frac{2^{1/v}\lambda \left[ ((v-1)/v) \log \log n + \log \left\{
2\Gamma\left( 1/v\right) \right\} \right]} {v\left(\log n\right)^{1 - 1/v}}.
\end{cases}
\end{eqnarray}
Peng et al. (2009) showed that
\begin{equation}\label{eq1.3}
\frac{M_{n}-\beta_{n}}{\alpha_{n}}\stackrel{d}{\to}M,
\end{equation}
where $M$ follows the Gumbel extreme value distribution $\Lambda(x)$. With normalizing constants $\alpha^{*}_{n}$ and $\beta^{*}_{n}$ given by
\begin{equation}\label{eq1.5}
\begin{cases}
\alpha^{*}_{n}=p\alpha_{n}\beta_{n}^{p-1};\\
\beta^{*}_{n}=\beta_{n}^{p}.
\end{cases}
\end{equation}
We will show that \eqref{add0} holds  and investigate further the higher-order expansions and convergence rates of \eqref{add0}. Similarly, with $p\ne v$ the optimal convergence rates of $|M_{n,r}|^{p}$ is derived under the following normalizing constants
\begin{equation}\label{eq1.6}
c_{n}=2pv^{-1}\lambda^{v}b_{n}^{p-v},\quad d_{n}=b_{n}^{p},
\end{equation}
where constant $b_{n}$ is the solution of the equation
\begin{equation}\label{eq1.7}
2^{1/v}\lambda ^{1-v}
\Gamma(1/v)b_{n}^{v-1}\exp\left(\frac{b_{n}^{v}}{2\lambda^{v}}\right)=n.
\end{equation}
Note that for the normal case, it follows from \eqref{eq1.6} that $c_{n}=pb_{n}^{p-2}, d_{n}=b_{n}^{p}$ since $v=2, \lambda=1$, which are just the normalizing constants given by Hall (1980).

For the normal case, Hall (1980) showed that the optimal convergence rate of $M_{n,r}^{2}$ is the order of $1/(\log n)^{2}$ if we choose the normalizing constants $c_{n}^{*}=2(1-b_{n}^{-2})$ and $d_{n}^{*}=b_{n}^{2}-2b_{n}^{-2}$. For the powered $r$th largest order statistics $|M_{n,r}|^{p}$ from the $\GED(v)$ sample, it follows from \eqref{eq1.6} that the convergence rate can be improved if $p=v$. Some techniques are used here to find the optimal normalizing constants $c_{n}^{*}$ and $d_{n}^{*}$ as $p=v$. Details are given as follows.

By  Eq. (3.1) of Lemma 1 in Jia and Li (2014), for $v\in(0,1)\cup(1,+\infty)$ and large $x$ we have
\begin{eqnarray}\label{eq1.8}
1-G_{v}(x)&=&\frac{2\lambda^{v}}{v}\Big\{1+2(v^{-1}-1)\lambda^{v}x^{-v}+4(v^{-1}-1)(v^{-1}-2)\lambda^{2v}x^{-2v}\nonumber\\
&&+8(v^{-1}-1)(v^{-1}-2)(v^{-1}-3)\lambda^{3v}x^{-3v}+O(x^{-4v})\Big\}x^{1-v}g_{v}(x),
\end{eqnarray}
where $X\sim G_{v}(x)$, the $\GED(v)$ distribution function. For the special case of $p=v$, let $F(x)=\P\{|X|^{v}\le x\}$ denote the distribution of $|X|^{v}$, we have $1-F(x)=2\{1-G_{v}(x^{1/v})\}$. Without loss of generality, assume that $1+2(v^{-1}-1)\lambda^{v}\neq0$. It follows from \eqref{eq1.8} that
\begin{align}\label{eq1.9}
1-F(x)&=\frac{2^{1-1/v}\lambda^{v-1}}{\Gamma(1/v)}\Big\{1+2(v^{-1}-1)\lambda^{v}x^{-1}+4(v^{-1}-1)(v^{-1}-2)\lambda^{2v}x^{-2}\nonumber\\
&\quad+8(v^{-1}-1)(v^{-1}-2)(v^{-1}-3)\lambda^{3v}x^{-3}+O(x^{-4})\Big\}x^{1/v-1}\exp\left(-\frac{x}{2\lambda^{v}}\right)\nonumber\\
&= \frac{2^{1-1/v}\lambda^{v-1}}{\Gamma(1/v)}\Big\{1+4(v^{-1}-1)(v^{-1}-2)\lambda^{2v}x^{-2}-16(v^{-1}-1)(v^{-1}-2)\lambda^{3v}x^{-3}
+O(x^{-4})\Big\}\nonumber\\
&\quad\times\left(1+2(v^{-1}-1)\lambda^{v}x^{-1} \right)x^{1/v-1}\exp\left(-\frac{x}{2\lambda^{v}}\right)\nonumber\\
&=\frac{2^{1-1/v}\lambda^{v-1}\left(1+2(v^{-1}-1)\lambda^{v}\right)e^{\frac{1}{2\lambda^{v}}}}{\Gamma(1/v)}
\Big\{1+4(v^{-1}-1)(v^{-1}-2)\lambda^{2v}x^{-2}\nonumber\\
&\quad-16(v^{-1}-1)(v^{-1}-2)\lambda^{3v}x^{-3}+O(x^{-4})\Big\}\exp\left(-\int_{1}^{x}\frac{g(t)}{f(t)}dt\right),
\end{align}
where
\begin{equation}\label{eq1.10}
\begin{cases}
g(t)=1-4(v^{-1}-1)(v^{-1}-2)\lambda^{2v}t^{-2}\to 1\; \mbox{as}\; t\to\infty;\\
f(t)= 2\lambda^{v}\left(1+2\lambda^{v}(v^{-1}-1)t^{-1}\right)\;\mbox{with}\; f^{\prime}(t)\to 0\; \mbox{as}\; t\to\infty.
\end{cases}
\end{equation}
Now similar to Hall (1980), we choose
\begin{equation}\label{eq1.11}
\begin{cases}
d_{n}^{*}=b^{v}_{n}+4(v^{-1}-1)\lambda^{2v}b_{n}^{-v}\\
c_{n}^{*}=f(b_{n}^{v})=2\lambda^{v}+4(v^{-1}-1)\lambda^{2v}b_{n}^{-v},
\end{cases}
\end{equation}
where $b_{n}$ is given by \eqref{eq1.7}. Note that if $v=2,\lambda=1$, the normal case,  $c_{n}^{*}=2-2b_{n}^{-2}$, $d_{n}^{*}=b_{n}^{2}-2b_{n}^{-2}$, just the normalizing constants given by Hall (1980).

Rest of this paper is organized as follows. Section \ref{sec2} provides the main results. Proofs of the main results with
some auxiliary lemmas are deferred to Section \ref{sec3}.

\section{Main results}\label{sec2}
In this section, we provide the higher-order distributional expansions of powered order statistics under different normalizing constants. Throughout this paper,  let $\Lambda_{r}(x)=\Lambda(x)\sum_{j=0}^{r-1}e^{-jx}/j!$
for positive integer $r$ and $ \Lambda_{r}(x)=0 $ for $ r\leq 0$.

\begin{theorem}\label{thm1}
Let $\{X_n,n\ge 1\}$ be a sequence of independent random variables with common distribution $G_v(x), v > 0$, and $M_{n,r}$ denotes the $r$th largest order statistics of $\{X_{1}, X_{2}, \cdots, X_{n}\}$. Then,
\begin{itemize}
\item[(i)]~if $v=1$ and $p=1$, with normalizing constants $\alpha^{*}_{n}$ and $\beta^{*}_{n}$ given by $\alpha^{*}_{n}=2^{-\frac{1}{2}}$, $\beta^{*}_{n}=2^{-\frac{1}{2}}\log \frac{n}{2}$, we have
\begin{eqnarray*}
  &&\lim_{n\to\infty}n\Big\{n\Big[ \P\left(|M_{n,r}|^p\leq \alpha_{n}^{*}x+\beta_{n}^{*}\right)-\Lambda_{r}(x)\Big]-\frac{e^{-(r+1)x}\left[(r-1)e^{x}-1\right]}{2(r-1)!}\Lambda(x)\Big\}\\
  &&=\frac{e^{-(r+2)x}}{24(r-1)!}\left[(-3r^{3}+10r^{2}-9r+2)e^{2x}+(9r^{2}-11r+2)e^{x}+3e^{-x}-9r+1\right]\Lambda(x).
  \end{eqnarray*}
\item[(ii)]~if $v=1$ and $p\neq 1$, with normalizing constants $\alpha^{*}_{n}$ and $\beta^{*}_{n}$ given by $\alpha^{*}_{n}=p2^{-\frac{p}{2}}(\log \frac{n}{2})^{p-1}$, $\beta^{*}_{n}=(2^{-\frac{1}{2}}\log \frac{n}{2})^p$, we have
\begin{eqnarray*}
  &&\lim_{n\to\infty}(\log n)\Big\{(\log\frac{n}{2})\Big[ \P\left(|M_{n,r}|^p\leq \alpha_{n}^{*}x+\beta_{n}^{*}\right)-\Lambda_{r}(x)\Big]-\frac{(1-p)x^2e^{-rx}}{2(r-1)!}\Lambda(x)\Big\}\\
  &&=\frac{(1-p)x^{3}e^{-rx}}{24(r-1)!}\left[4(1-2p)-3(1-p)rx+3(1-p)xe^{-x}\right]\Lambda(x).
  \end{eqnarray*}
\item[(iii)]~if $v\in(0,1)\bigcup(1,+\infty)$ and $p>0$, with normalizing constants $\alpha^{*}_{n}$ and $\beta^{*}_{n}$ given by \eqref{eq1.5}, we have
 \begin{eqnarray*}
   &&\lim_{n\to\infty}(\log\log n)\Big\{(\frac{\log n}{(\log\log n)^2})\big[\P\left(|M_{n,r}|^{p}\le \alpha_{n}^{*}x+\beta_{n}^{*}\right)-\Lambda_{r}(x)\big]-\frac{(1-v^{-1})^{3}e^{-rx}}{2(r-1)!}\Lambda(x)\Big\}\\
  &&=-(1-v^{-1})^{2}\left(1-\log2\Gamma(\frac{1}{v})+x\right)\frac{e^{-rx}}{(r-1)!}\Lambda(x).
  \end{eqnarray*}
\end{itemize}
\end{theorem}

\begin{theorem}\label{thm2}
Let $M_{n,r}$ denotes the $r$th  largest order statistics of $\{X_{k},1\le k\le n\} $, a sample from the $\GED(v)$ distribution, then
\begin{itemize}
\item[(i)]~if $v\in(0,1)\bigcup(1,+\infty)$ and $p\neq v$, with normalizing constants $c_{n}$ and $d_{n}$ given by \eqref{eq1.6}, we have
  \begin{eqnarray*}
   &&\lim_{n\to\infty}b_n^v\bigg\{b_n^v\Big[\P\left(|M_{n,r}|^p\leq c_nx+d_n\right)-\Lambda_{r}(x)\big]-\Lambda(x)h_v(x)\frac{e^{-(r-1)x}}{(r-1)!}\bigg\}\\
   &=&\left[q_v(x)+\left(1-(r-1)e^{x}\right)\frac{h_v^2(x)}{2}\right]\frac{e^{-(r-1)x}}{(r-1)!}\Lambda(x),
   \end{eqnarray*}
where
\begin{eqnarray}
\label{eq2.1}
h_v(x)=[v^{-1}(v-p)\lambda^vx^2-2v^{-1}(1-v)\lambda^vx-2(v^{-1}-1)\lambda^v]e^{-x}
\end{eqnarray}
and
\begin{eqnarray}
\label{eq2.2}
q_v(x)&=&\Big[-\frac{1}{2}\lambda^{2v}v^{-2}(v-p)^{2}x^{4}+v^{-2}(v-p)\lambda^{2v}(2-\frac{4}{3}v-\frac{4}{3}p)x^3-2v^{-2}(1-v)\lambda^{2v}x^2\nonumber\\
&&-4(v^{-1}-1)(v^{-1}-2)\lambda^{2v}x-4(v^{-1}-1)(v^{-1}-1)\lambda^{2v}\Big]e^{-x}.
\end{eqnarray}
\item[(ii)]~if $v\in(0,1)\bigcup(1,+\infty)$ and $p=v$, with normalizing constants $c^{*}_{n}$ and $d^{*}_{n}$ given by \eqref{eq1.11}, we have
 \begin{eqnarray*}
   \lim_{n\to\infty}b_n^{v}\bigg\{b_n^{2v}\Big[\P\left(|M_{n,r}|^v\leq c_n^*x+d_n^*\right)-\Lambda_{r}(x)\Big]-\Lambda(x)s_{v}(x)\frac{e^{-(r-1)x}}{(r-1)!}\bigg\}=\Lambda(x)b_{v}(x)\frac{e^{-(r-1)x}}{(r-1)!},
  \end{eqnarray*}
where
\begin{eqnarray}
\label{eq2.3}
s_{v}(x)=2(v^{-1}-1)\lambda^{2v}\bigg[x^2-2(v^{-1}-2)x-(3v^{-1}-5)\bigg]e^{-x},
\end{eqnarray}
and
\begin{eqnarray}
\label{eq2.4}
b_{v}(x)&=&-\frac{4}{3}(v^{-1}-1)\lambda^{3v}\Big[(4-v^{-1})(v^{-1}-1)x^3-6(v^{-1}-2)x^2\nonumber\\
&&-6(3v^{-1}-5)x+(2v^{-2}-22v^{-1}+32)\Big]e^{-x}.
\end{eqnarray}
\end{itemize}
\end{theorem}

\begin{remark}\label{remark1}
Theorem \ref{thm1} (i)-(ii) shows the difference of the optimal convergence rates for the powered order statistics of the Laplace distribution as $p=1$ and $p\ne 1$, respectively. Meanwhile, it follows from \eqref{eq1.2}, \eqref{eq1.5}-\eqref{eq1.7} that $c_{n}=\alpha_{n}^{*},d_{n}=\beta_{n}^{*}$ as $v=1$ since $\lambda=2^{-3/2}$, so it is not necessary to consider the case of $v=1$ in Theorem \ref{thm2}.
\end{remark}



\begin{remark}\label{remark2}
For $p\not=v$, with normalizing constants $c_{n}$ and $d_{n}$ given by \eqref{eq1.6} Theorem \ref{thm2} shows that the optimal convergence rate of $\P(|M_{n,r}|^p\leq c_nx+d_n)$ to the extreme value distribution $\Lambda_{r}(x)$ is proportional to $1/ \log n$ since $b_{n}^{v}\sim 2\lambda^{v}\log n$ by \eqref{eq1.7}, while it can be improved to the order of $1/(\log n)^2$ with optimal choice of normalizing constants $c_{n}^{*}$ and $d_{n}^{*}$ given by \eqref{eq1.11} as $p=v$, which coincides with the normal case studied by Hall (1980).
\end{remark}

\section{The proofs}\label{sec3}
In order to prove the main results, we need some auxiliary lemmas.
\begin{lemma}\label{lemma1}
Let $G_v(x)$ denote the $\GED(v)$ distribution function with parameter $v>0$, then
\begin{itemize}
\item[(i)]~if $v=1$ and $p=1$, with normalizing constants $\alpha^{*}_{n}$ and $\beta^{*}_{n}$ given by $\alpha^{*}_{n}=2^{-\frac{1}{2}}$, $\beta^{*}_{n}=2^{-\frac{1}{2}}\log \frac{n}{2}$, for $\alpha^{*}_{n}x+\beta^{*}_{n}>0$ we have
\begin{eqnarray}
\label{eq3.2}
1-G_1\big(\alpha^*_nx+\beta^*_n\big)=n^{-1} e^{-x}.
\end{eqnarray}

\item[(ii)]~if $v=1$ and $p\neq 1$, with normalizing constants $\alpha^{*}_{n}$ and $\beta^{*}_{n}$ given by $\alpha^{*}_{n}=p2^{-\frac{p}{2}}(\log \frac{n}{2})^{p-1}$, $\beta^{*}_{n}=(2^{-\frac{1}{2}}\log \frac{n}{2})^p$, for large $n$ we have
\begin{align}
\label{eq3.1}
1-G_1\big((\alpha^*_nx+\beta^*_n)^{\frac{1}{p}}\big)
=n^{-1}e^{-x}\left\{1-\frac{(1-p)x^2}{2\log \frac{n}{2}}+\frac{(1-p)[3(1-p)x-4(1-2p)]x^{3}}{24(\log \frac{n}{2})^{2}}+o(\frac{1}{(\log \frac{n}{2})^{2}})\right\}.
\end{align}

\item[(iii)]~if $v\in(0,1)\bigcup(1,+\infty)$, with normalizing constants $\alpha^{*}_{n}$ and $\beta^{*}_{n}$ given by \eqref{eq1.5}, for large $n$ we have
\begin{align}
\label{eq3.3}
1-G_v\big((\alpha^*_nx+\beta^*_n)^{\frac{1}{p}}\big)=&n^{-1}e^{-x}\left\{1-\frac{(1-v^{-1})^3(\log\log n)^{2}}{2\log n}\right.\nonumber\\
&\left.+\frac{(1-v^{-1})^{2}(1-\log 2\Gamma({1}/{v}) +x)\log\log n}{\log n}+o\big(\frac{\log\log n}{\log n}\big)\right\}.
\end{align}
\end{itemize}
\end{lemma}
\begin{proof}~~(i). If $v=1$ and $p=1$, by \eqref{eq1.1} the Laplace distribution function $G_{1}(x)$ is given by
\begin{eqnarray}
\label{eq3.6}
G_{1}(x)=1-\frac{1}{2}\exp(-\sqrt{2}x)\;\;\mbox{as}\;\; x>0.
\end{eqnarray}
Putting the values of $\alpha_{n}^{*}$ and $\beta_{n}^{*}$ given by Theorem \ref{thm1}(i) into \eqref{eq3.6}, we have
\[
1-G_1\big(\alpha^*_nx+\beta^*_n\big)= n^{-1}e^{-x}\]
as $\alpha^{*}_{n}x+\beta^{*}_{n}>0$.

(ii). If $v=1$ and $p\neq 1$,
note that
\begin{eqnarray}
\label{eq3.7}
(\alpha^*_nx+\beta^*_n)^{\frac{1}{p}}
=\frac{\log\frac{n}{2}}{\sqrt2}\left(1+\frac{x}{\log\frac{n}{2}}+\frac{(1-p)x^2}{2(\log\frac{n}{2})^2}
+\frac{(1-p)(1-2p)x^3}{6(\log\frac{n}{2})^3}+o(\frac{1}{(\log\frac{n}{2})^3})\right).
\end{eqnarray}
The claimed result \eqref{eq3.1} follows from \eqref{eq3.6} and \eqref{eq3.7}.

(iii). If $v\in(0,1)\bigcup(1,+\infty)$, with $z_{n,p}(x)=(\alpha_{n}^{*}x+\beta_{n}^{*})^{{1}/{p}}$ with normalizing constants $\alpha^{*}_{n}$ and $\beta^{*}_{n}$ given by \eqref{eq1.5}, we have
\begin{eqnarray}
\label{eq3.8}
\nonumber z_{n,p}^{1-v}(x)&=&2^{\frac{1-v}{v}}\lambda^{1-v}(\log n)^{\frac{1-v}{v}}\Big\{1+\frac{(v^{-1}-1)[x-\log2\Gamma(\frac{1}{v})]+(v^{-1}-1)^2(\log\log n)}{\log n}\\
&&-\frac{(v^{-1}-1)^3(\log\log n)^2}{2(\log n)^{2}}
+o\big((\frac{\log\log n}{\log n})^2\big)\Big\}
\end{eqnarray}
and
\begin{align}
\label{eq3.9}
\nonumber&g_v\big(z_{n,p}(x)\big)\\
\nonumber=&\frac{ve^{-x}}{\lambda2^{\frac{1}{v}}}n^{-1}(\log n)^{\frac{v-1}{v}}\exp\Big[-\frac{\frac{(1-v^{-1})^{3}}{2}(\log\log n)^2+(1-v^{-1})^{2}\left(\log 2\Gamma(\frac{1}{v}) -x\right)\log\log n}{\log n}+o\big(\frac{\log\log n}{\log n}\big)\Big]\\
=&\frac{ve^{-x}}{\lambda2^{\frac{1}{v}}}n^{-1}(\log n)^{\frac{v-1}{v}}\Big[1-\frac{\frac{(1-v^{-1})^{3}}{2}(\log\log n)^2+(1-v^{-1})^{2}\left(\log 2\Gamma(\frac{1}{v}) -x\right)\log\log n}{\log n}+o\big(\frac{\log\log n}{\log n}\big)\Big].
\end{align}
Further,
\begin{eqnarray}
\label{eq3.10}
\nonumber&&1+2(v^{-1}-1)\lambda^v\big(z_{n,p}(x)\big)^{-v}+4(v^{-1}-1)(v^{-1}-2)\lambda^{2v}\big(z_{n,p}(x)\big)^{-2v}\\
&=&1+\frac{v^{-1}-1}{\log n}-\frac{(v^{-1}-1)^{2}\log\log n}{\log^2n}+o(\frac{\log\log n}{\log^2n}).
\end{eqnarray}
Combining \eqref{eq1.8} and \eqref{eq3.8}-\eqref{eq3.10}, we derive \eqref{eq3.3}.
\end{proof}

\begin{lemma}\label{lemma2}
Let $G_v(x)$ denote the $\GED(v)$ distribution function with parameter $v>0$, then
\begin{itemize}
\item[(i)]~if $v\in(0,1)\bigcup(1,+\infty)$ and $p\neq v$, with normalizing constants $c_{n}$ and $d_{n}$ given by  \eqref{eq1.6}, we have
\begin{eqnarray}
\label{eq3.4}
\nonumber&&1-G_v\big((c_nx+d_n)^{\frac{1}{p}}\big)\\
&=&n^{-1}e^{-x}\Big\{1-\Big[v^{-1}(v-p)\lambda^vx^2-2v^{-1}(1-v)\lambda^vx-2(v^{-1}-1)\lambda^v\Big]b_{n}^{-v}\nonumber\\
&&+\Big[\frac{1}{2}\lambda^{2v}v^{-2}(v-p)^{2}x^{4}-v^{-2}(v-p)\lambda^{2v}(2-\frac{4}{3}v-\frac{4}{3}p)x^3+2v^{-2}(1-v)\lambda^{2v}x^2\nonumber\\
&&+4(v^{-1}-1)(v^{-1}-2)\lambda^{2v}x+4(v^{-1}-1)(v^{-1}-2)\lambda^{2v}\Big]b_{n}^{-2v}+o(b_n^{-2v})\Big\}.
\end{eqnarray}

\item[(ii)]~if $v\in(0,1)\bigcup(1,+\infty)$ and $p=v$, with normalizing constants $c^*_{n}$ and $d^*_{n}$ given by (\ref{eq1.11}), we have
\begin{eqnarray}
\label{eq3.5}
&&1-G_v\big((c^*_nx+d^*_n)^{\frac{1}{p}}\big)\nonumber\\
&=&n^{-1}e^{-x}\Big\{1 -2(v^{-1}-1)\lambda^{2v}[x^2-2(v^{-1}-2)x-3v^{-1}+5]{b_{n}^{-2v}}\nonumber\\
&&+(v^{-1}-1)\lambda^{3v}[\frac{4}{3}(v^{-1}-1)(4-v^{-1})x^3-8(v^{-1}-2)x^2-8(3v^{-1}-5)x]{b_n^{-3v}}\nonumber\\
&&+\frac{8}{3}(v^{-1}-1)(v^{-2}-11v^{-1}+16)\lambda^{3v}{b_n^{-3v}}+o({b_n^{-3v}})\Big\}.
\end{eqnarray}
\end{itemize}
\end{lemma}

\begin{proof}
(i), if $v\in(0,1)\bigcup(1,+\infty)$ and $p\neq v$, with normalizing constants $c_n$ and $d_n$, we have
\begin{eqnarray*}
z_{n,p}(x)=(c_nx+d_n)^{\frac{1}{p}}=b_n\left(1+\frac{2pv^{-1}\lambda^vx}{b^v_n}\right)^{\frac{1}{p}}.
\end{eqnarray*}
By arguments similar to \eqref{eq3.8}-\eqref{eq3.10}, by \eqref{eq1.6} and \eqref{eq1.7} we have
\begin{eqnarray}
\label{eq3.11}
z^{1-v}_{n,p}(x)=b_n^{1-v}\left(1+\frac{2v^{-1}(1-v)\lambda^vx}{b_n^v}+\frac{2(1-v)(1-v-p)v^{-2}\lambda^{2v}x^2}{b_n^{2v}}+o(\frac{1}{b_n^{2v}})\right)
\end{eqnarray}
and
\begin{eqnarray}
\label{eq3.12}
g_v\big(z_{n,p}(x)\big)&=&n^{-1}e^{-x}\frac{v}{2\lambda^v}b_n^{v-1}\left\{1-\frac{v^{-1}(v-p)\lambda^vx^2}{b_n^v}\right.\nonumber\\
&&\left.-\frac{4\lambda^{2v}v^{-2}(v-p)(v-2p)x^3-3\lambda^{2v}v^{-2}(v-p)^{2}x^{4}}{6b_n^{2v}}
+o(\frac{1}{b_n^{2v}})\right\}.
\end{eqnarray}
Further,
\begin{eqnarray}
\label{eq3.13}
\nonumber&&1+2(v^{-1}-1)\lambda^v\big(z_{n,p}(x)\big)^{-v}+4(v^{-1}-1)(v^{-1}-2)\lambda^{2v}\big(z_{n,p}(x)\big)^{-2v}\\
&=&1+\frac{2(v^{-1}-1)\lambda^v}{b_n^v}+\frac{4(v^{-1}-1)(v^{-1}-2)\lambda^{2v}-4(v^{-1}-1)\lambda^{2v}x}{b_n^{2v}}+o(\frac{1}{b_n^{2v}}).
\end{eqnarray}
Combining \eqref{eq1.8} and \eqref{eq3.11}-\eqref{eq3.13}, we derive the desired result \eqref{eq3.4}.

(ii). If $v\in(0,1)\bigcup(1,+\infty)$ and $p=v$, with normalizing constants $c^*_n$ and $d^*_n$ given by \eqref{eq1.11}, let
\begin{eqnarray*}
z_{n,v}(x)=(c^*_nx+d^*_n)^{\frac{1}{v}}=b_n\left(1+\frac{2\lambda^v}{b^v_n}x+\frac{4(v^{-1}-1)\lambda^{2v}(x+1)}{b_n^{2v}}\right)^{\frac{1}{v}}.
\end{eqnarray*}
Arguments similar to (\ref{eq3.8})-(\ref{eq3.10}), we can get
\begin{eqnarray}
\label{eq3.14}
\nonumber\big(z_{n,v}(x)\big)^{1-v}&=&b_n^{1-v}\left(1+\frac{2(v^{-1}-1)\lambda^{v}x}{b^v_n}
+\frac{4(v^{-1}-1)^2\lambda^{2v}(x+1)+2(v^{-1}-1)(v^{-1}-2)\lambda^{2v}x^2}{b_n^{2v}}\right.\\
&&\left.+\frac{8(v^{-1}-1)^2(v^{-1}-2)\lambda^{3v}x(x+1)}{b_n^{3v}}+o(\frac{1 }{b_n^{3v}})\right),
\end{eqnarray}
and by \eqref{eq1.7},
\begin{eqnarray}
\label{eq3.15}
\nonumber g_v\big(z_{n,v}(x)\big)&=&\frac{ve^{-x}}{2\lambda^v}n^{-1}b_n^{v-1}\left\{1-\frac{2(v^{-1}-1)\lambda^v(x+1)}{b_n^v}
+\frac{2(v^{-1}-1)^2\lambda^{2v}(x+1)^2}{b_n^{2v}}\right.\\
&&\left.-\frac{4(v^{-1}-1)^3\lambda^{3v}(x+1)^3}{3b_n^{3v}}+o(\frac{1}{b_n^{3v}})\right\}.
\end{eqnarray}
Further,
\begin{eqnarray}
\label{eq3.16}
&&1+2(v^{-1}-1)\lambda^v\big(z_{n,v}(x)\big)^{-v}+4(v^{-1}-1)(v^{-1}-2)\lambda^{2v}\big(z_{n,v}(x)\big)^{-2v} \nonumber\\
&&+8(v^{-1}-1)(v^{-1}-2)(v^{-1}-3)\lambda^{3v}\big(z_{n,v}(x)\big)^{-3v}\nonumber\\
&=&1+\frac{2(v^{-1}-1)\lambda^v}{b_n^v}+\frac{4(v^{-1}-1)(v^{-1}-2)\lambda^{2v}-4(v^{-1}-1)\lambda^{2v}x}{b_n^{2v}}\nonumber\\
&&+\frac{8(v^{-1}-1)\lambda^{3v}x^2-8(v^{-1}-1)(3v^{-1}-5)\lambda^{3v}x}{b_n^{3v}}\nonumber\\
&&+\frac{8(v^{-1}-1)\lambda^{3v}(v^{-2}-6v^{-1}+7)}{b_n^{3v}}+o(\frac{1}{b_n^{3v}}).
\end{eqnarray}
Combining \eqref{eq1.8} and \eqref{eq3.14}-\eqref{eq3.16}, we derive \eqref{eq3.5}.
\end{proof}

\begin{lemma}\label{lemma3}
Let $\{X_n,n\geq1\}$ be a sequence of i.i.d. random variables with common distribution $G_v(x)$ with parameter $v > 0$ and $M_{n,r}$ denotes the $r$th largest order statistics of $\{X_{1}, X_{2}, \cdots, X_{n} \}$. Assume that there exists positive constant $z_{n}(x)$ such that $n(1-G_{v}(z_{n}(x))\to e^{-x}$, then
\begin{eqnarray}
\label{eq3.17}
&&\P\big(|M_{n,r}|^p\leq z_{n}^{p}(x)\big) -\Lambda_{r}(x)\nonumber\\
&=&\Lambda(x)\left[1-\frac{1}{2}(1-\theta_{n,v}(x))(r-1-e^{-x})\right](1-\theta_{n,v}(x))\frac{e^{-rx}}{(r-1)!}+O(n^{-1}),
\end{eqnarray}
where $ \theta_{n,v}(x)=ne^{x}\left(1-G_{v}(z_{n}(x))\right)$.
\end{lemma}

\begin{proof} First, note that
\begin{eqnarray}
\label{add1}
\Lambda(x)\sum_{j=0}^{r-1}\frac{je^{-jx}}{j!}=e^{-x}\Lambda_{r-1}(x)
\end{eqnarray}
and
\begin{eqnarray}
\label{add2}
\Lambda(x)\sum_{j=0}^{r-1} \frac{j^{2}e^{-jx}}{j!}
&=&e^{-2x}\Lambda_{r-2}(x)+e^{-x}\Lambda_{r-1}(x).
\end{eqnarray}
By arguments similar to Hall (1980) and some tedious calculation, we have
\begin{eqnarray*}
&&\P\big(|M_{n,r}|^p\leq z_{n}^{p}(x)\big) -\Lambda_{r}(x)\\
&=&\sum_{j=0}^{r-1}\binom n j \left[1-n^{-1}e^{-x}\theta_{n,v}(x)\right]^{n-j}\left[n^{-1}e^{-x}\theta_{n,v}(x)\right]^{j}-\Lambda_{r}(x)+O(n^{r-1}2^{-n})\\
&=&\sum_{j=0}^{r-1}\left(1-n^{-1}e^{-x}\theta_{n,v}(x)\right)^{n}\theta^{j}_{n,v}(x)\frac{e^{-jx}}{j!}-\Lambda_{r}(x)+O(n^{-1})\\
&=&\Lambda(x)\sum_{j=0}^{r-1}\left\{\exp\left[(1-\theta_{n,v}(x))e^{-x}\right]\right\}\left[1-(1-\theta_{n,v}(x))\right]^{j}\frac{e^{-jx}}{j!}-\Lambda_{r}(x)+O(n^{-1})\\
&=&\Lambda(x)\sum_{j=0}^{r-1}\left[1+(1-\theta_{n,v}(x))e^{-x}+\frac{(1-\theta_{n,v}(x))^{2}}{2}e^{-2x}\right]\\
&&\times\left[1-j(1-\theta_{n,v}(x))+\frac{j(j-1)}{2}(1-\theta_{n,v}(x))^{2}\right]\frac{e^{-jx}}{j!}-\Lambda_{r}(x)+O(n^{-1})\\
&=&\Lambda(x)\sum_{j=0}^{r-1}\Big\{1+(e^{-x}-j)(1-\theta_{n,v}(x))\\
&&+\left(\frac{j(j-1)}{2}-je^{-x}+\frac{e^{-2x}}{2}\right)(1-\theta_{n,v}(x))^{2}\Big\}\frac{e^{-jx}}{j!}-\Lambda_{r}(x)+O(n^{-1})\\
&=&(1-\theta_{n,v}(x))\left(\Lambda_{r}(x)-\Lambda_{r-1}(x)\right)e^{-x}\nonumber\\
&&-\frac{1}{2}(1-\theta_{n,v}(x))^{2}
\Big[2\Lambda_{r-1}(x)-\Lambda_{r-2}(x)-\Lambda_{r}(x)\Big]e^{-2x}+O(n^{-1})\\
&=&\Lambda(x)(1-\theta_{n,v}(x))\frac{e^{-rx}}{(r-1)!}-\frac{1}{2}\Lambda(x)(1-\theta_{n,v}(x))^{2}\frac{e^{-(r+1)x}[(r-1)e^{x}-1]}{(r-1)!}+O(n^{-1})\\
&=&\Lambda(x)\left[1-\frac{1}{2}(1-\theta_{n,v}(x))(r-1-e^{-x})\right](1-\theta_{n,v}(x))\frac{e^{-rx}}{(r-1)!}+O(n^{-1}).
\end{eqnarray*}
The desired result follows.
\end{proof}

\noindent{\bf Proof of Theorem \ref{thm1}.}
(i). Note that Lemma \ref{lemma3} shows that $\P\left(|M_{n,r}|\leq z_{n,1}(x)\right)-\Lambda_{r}(x)=O(1/n)$ since $\theta_{n,1}(x)=ne^{x}(1-G_{1}(z_{n,1}(x))=1$ by Lemma \ref{lemma1}(i) in the case of $v=p=1$, where $z_{n,1}(x)=\alpha_{n}^{*}x+\beta_{n}^{*}$ with $\alpha_{n}^{*}$ and $\beta_{n}^{*}$ given by Theorem \ref{thm1}(i). Higher-order expansions are needed here. By using \eqref{eq3.2},\eqref{add1}, \eqref{add2} and the following two facts
\[
\Lambda(x)\sum_{j=0}^{r-1} \frac{j^{3}e^{-jx}}{j!}=e^{-3x}\Lambda_{r-3}(x)+3e^{-2x}\Lambda_{r-2}(x)+e^{-x}\Lambda_{r-1}(x)
\]
and
\[
\Lambda(x)\sum_{j=3}^{r-1}\sum_{i=1}^{j-2}\sum_{k=i+1}^{j-1}\frac{ik e^{-jx}}{j!}=\frac{e^{-4x}}{8}\Lambda_{r-4}(x)+\frac{e^{-3x}}{3}\Lambda_{r-3}(x),
\]
we have
\begin{eqnarray}\label{eq4.3}
&&\P\left(|M_{n,r}|\leq z_{n,1}(x)\right)-\Lambda_{r}(x)\nonumber\\
&=&\sum_{j=0}^{r-1}\frac{n(n-1) \ldots (n-j+1)}{j!}\left(1-\frac{e^{-x}}{n}\right)^{n-j}(\frac{e^{-x}}{n})^{j}-\Lambda_{r}(x)+O(n^{r-1}2^{-n})\nonumber\\
&=&\Lambda(x)\sum_{j=0}^{r-1}\left(1-\frac{j(j-1)}{2n}+n^{-2}\sum_{i=1}^{j-2}\sum_{k=i+1}^{j-1}ik+o(n^{-2})\right)\nonumber\\
&&\times\Big[1+\frac{2j-e^{-x}}{2n}e^{-x}+\frac{[4(3j-2e^{-x})+3(2j-e^{-x})^{2}]e^{-2x}}{24n^{2}}+o(n^{-2})\Big]\frac{e^{-jx}}{j!}-\Lambda_{r}(x)+o(n^{-2})\nonumber\\
&=&\Lambda(x)\sum_{j=0}^{r-1}\Big\{\left[-\frac{j(j-1)}{2}+\frac{1}{2}(2j-e^{-x})e^{-x}\right]n^{-1}+\Big[\sum_{i=1}^{j-2}\sum_{k=i+1}^{j-1}ik-\frac{j(j-1)}{4}(2j-e^{-x})e^{-x}\nonumber\\
&&+\left(\frac{1}{6}(3j-2e^{-x})+\frac{1}{8}(2j-e^{-x})^{2}\right)e^{-2x}\Big]n^{-2}\Big\}\frac{e^{-jx}}{j!}+o(n^{-2})\nonumber\\
&=&\frac{1}{2n}\Big[\Lambda_{r-1}(x)-\Lambda_{r-2}(x)-\left(\Lambda_{r}(x)-\Lambda_{r-1}(x)\right)\Big]e^{-2x}+\frac{1}{n^{2}}\Big[\frac{e^{-4x}}{8}\Lambda_{r-4}(x)
+(\frac{1}{3}-\frac{e^{-x}}{2})e^{-3x}\Lambda_{r-3}(x)\nonumber\\
&&+(\frac{3e^{-x}}{4}-1)e^{-3x}\Lambda_{r-2}(x)+(1-\frac{e^{-x}}{2})e^{-3x}\Lambda_{r-1}(x)+(\frac{e^{-x}}{8}-\frac{1}{3})e^{-3x}\Lambda_{r}(x)\Big]+o(n^{-2})\nonumber\\
&=&\frac{1}{2n}\Big[\Lambda_{r-1}(x)-\Lambda_{r-2}(x)-\left(\Lambda_{r}(x)-\Lambda_{r-1}(x)\right)\Big]e^{-2x}-\frac{1}{n^{2}}\Big[\frac{e^{-4x}}{8}\left(\Lambda_{r-3}(x)-\Lambda_{r-4}(x)\right)\nonumber\\
&&-\frac{e^{-3x}}{24}(9e^{-x}-8)\left(\Lambda_{r-2}(x)-\Lambda_{r-3}(x)\right)+\frac{e^{-3x}}{24}(9e^{-x}-16)\left(\Lambda_{r-1}(x)-\Lambda_{r-2}(x)\right)\nonumber\\
&&-\frac{e^{-3x}}{24}(3e^{-x}-8)\left(\Lambda_{r}(x)-\Lambda_{r-1}(x)\right)\Big]+o(n^{-2})\nonumber\\
&=&\frac{e^{-2x}}{2n}\Lambda(x)\left(\frac{e^{-(r-2)x}}{(r-2)!}-\frac{e^{-(r-1)x}}{(r-1)!}\right)-\frac{1}{n^{2}}\Big[\frac{e^{-4x}}{8}\frac{e^{-(r-4)x}}{(r-4)!}\nonumber\\
&&-\frac{e^{-3x}}{24}(9e^{-x}-8)\frac{e^{-(r-3)x}}{(r-3)!}+\frac{e^{-3x}}{24}(9e^{-x}-16)\frac{e^{-(r-2)x}}{(r-2)!}
-\frac{e^{-3x}}{24}(3e^{-x}-8)\frac{e^{-(r-1)x}}{(r-1)!}\Big]\Lambda(x)+o(n^{-2})\nonumber\\
&=&\Lambda(x)\frac{(r-1)e^{x}-1}{2n\{(r-1)!\}}e^{-(r+1)x}+\frac{e^{-(r+2)x}}{24n^{2}\{(r-1)!\}}\Big[(-3r^{3}+10r^{2}-9r+2)e^{2x}\nonumber\\
&&+(9r^{2}-11r+2)e^{x}+3e^{-x}-9r+1\Big]\Lambda(x)+o(n^{-2}).
\end{eqnarray}
Hence, it follows from \eqref{eq4.3} that
\begin{eqnarray*}
\lim_{n\to\infty}n\Big[ \P\left(|M_{n,r}|\leq \alpha_{n}^{*}x+\beta_{n}^{*}\right)-\Lambda_{r}(x)\Big]=\frac{e^{-(r+1)x}\left[(r-1)e^{x}-1\right]}{2(r-1)!}\Lambda(x)
\end{eqnarray*}
and
\begin{eqnarray*}
&&\lim_{n\to\infty}n\Big\{n\Big[ \P\left(|M_{n,r}|\leq \alpha_{n}^{*}x+\beta_{n}^{*}\right)-\Lambda_{r}(x)\Big]-\frac{e^{-(r+1)x}\left[(r-1)e^{x}-1\right]}{2(r-1)!}\Lambda(x)\Big\}\\
&&=\frac{e^{-(r+2)x}}{24(r-1)!}\left[(-3r^{3}+10r^{2}-9r+2)e^{2x}+(9r^{2}-11r+2)e^{x}+3e^{-x}-9r+1\right]\Lambda(x).
\end{eqnarray*}

(ii). In the case of $v=1, p\ne 1$, let $z_{n,p}(x)=(\alpha_{n}^{*}x+\beta_{n}^{*})^{1/p}$ with $\alpha_{n}^{*}$ and $\beta_{n}^{*}$ given by Theorem \ref{thm1}(ii). By using \eqref{eq3.1}, with $\theta_{n,p}(x)=ne^{x}(1-G_{1}(z_{n,p}(x))$ we have
\begin{align}
\label{eq4.1}
1-\theta_{n,p}(x)
=\frac{(1-p)x^2}{2\log \frac{n}{2}}-\frac{(1-p)[3(1-p)x-4(1-2p)]x^{3}}{24(\log \frac{n}{2})^{2}}+o(\frac{1}{(\log \frac{n}{2})^{2}}).
\end{align}
It follows from \eqref{eq4.1} and Lemma \ref{lemma3} that
\begin{eqnarray}
\label{add3}
&&\P\big(|M_{n,r}|^p\leq \alpha_{n}^{*}x+\beta_{n}^{*}\big) -\Lambda_{r}(x)\nonumber\\
&=&\Lambda(x)\left[1-\frac{1}{2}(1-\theta_{n,p}(x))(r-1-e^{-x})\right](1-\theta_{n,p}(x))\frac{e^{-rx}}{(r-1)!}+O(n^{-1})\nonumber\\
&=&\Lambda(x)\Big[\frac{(1-p)x^2}{2\log \frac{n}{2}}-\frac{(1-p)[3(1-p)x-4(1-2p)]x^{3}}{24(\log \frac{n}{2})^{2}}+o(\frac{1}{(\log \frac{n}{2})^{2}})\Big]\nonumber\\
&&\times\Big\{1-\frac{r-1-e^{-x}}{2}\Big[\frac{(1-p)x^2}{2\log \frac{n}{2}}-\frac{(1-p)[3(1-p)x-4(1-2p)]x^{3}}{24(\log \frac{n}{2})^{2}}+o(\frac{1}{(\log \frac{n}{2})^{2}})\Big]\Big\}\frac{e^{-rx}}{(r-1)!}+O(n^{-1})\nonumber\\
&=&\Lambda(x)\Big\{\frac{(1-p)x^2}{2\log\frac{n}{2}}+\frac{(1-p)x^{3}}{24(\log \frac{n}{2})^2}\Big[4(1-2p)-3(1-p)rx+3(1-p)xe^{-x}\Big]\Big\}\frac{e^{-rx}}{(r-1)!}+o\left(\frac{1}{(\log \frac{n}{2})^2}\right).\nonumber\\
\end{eqnarray}
Hence, following \eqref{add3} we have
\begin{eqnarray*}
\lim_{n\to\infty}(\log\frac{n}{2})\Big[ \P\left(|M_{n,r}|^p\leq \alpha_{n}^{*}x+\beta_{n}^{*}\right)-\Lambda_{r}(x)\Big]=\frac{(1-p)x^2e^{-rx}}{2(r-1)!}\Lambda(x)
\end{eqnarray*}
and
\begin{eqnarray*}
&&\lim_{n\to\infty}(\log n)\Big\{(\log\frac{n}{2})\Big[ \P\left(|M_{n,r}|^p\leq \alpha_{n}^{*}x+\beta_{n}^{*}\right)-\Lambda_{r}(x)\Big]-\frac{(1-p)x^2e^{-rx}}{2(r-1)!}\Lambda(x)\Big\}\\
&=&\frac{(1-p)x^{3}e^{-rx}\left[4(1-2p)-3(1-p)rx+3(1-p)xe^{-x}\right]}{24(r-1)!}\Lambda(x).
\end{eqnarray*}

(iii). For the case of $v\ne 1$ and $p>0$, let $z_{n,p}(x)=(\alpha_{n}^{*}x+\beta_{n}^{*})^{1/p}$ with $\alpha_{n}^{*}$ and $\beta_{n}^{*}$ given by Theorem \ref{thm1}(iii). With $\theta_{n,p}(x)=ne^{x}(1-G_{v}(z_{n,p}(x))$ we have
\begin{align}
\label{eq4.4}
1-\theta_{n,p}(x)=\frac{(1-v^{-1})^3(\log\log n)^{2}}{2\log n}-\frac{(1-v^{-1})^{2}(1+x-\log\{2\Gamma(\frac{1}{v})\})\log\log n}{\log n}+o\big(\frac{\log\log n}{\log n}\big)
\end{align}
due to \eqref{eq3.3}. Hence, it follows from \eqref{eq4.4} and Lemma \ref{lemma3} that
\begin{eqnarray*}
&&\P\big(|M_{n,r}|^p\leq \alpha_{n}^{*}x+\beta_{n}^{*}\big) -\Lambda_{r}(x)\nonumber\\
&=&\Lambda(x)\left[1-\frac{1}{2}(1-\theta_{n,p}(x))(r-1-e^{-x})\right](1-\theta_{n,p}(x))\frac{e^{-rx}}{(r-1)!}+O(n^{-1})\nonumber\\
&=&\Lambda(x)\Big[\frac{(1-v^{-1})^3(\log\log n)^{2}}{2\log n}-\frac{(1-v^{-1})^{2}(1-\log2\Gamma(\frac{1}{v})+x)\log\log n}{\log n}+o\big(\frac{\log\log n}{\log n}\big)\Big]\nonumber\\
&&\times\Big\{1-\frac{r-1-e^{-x}}{2}\Big[\frac{(1-v^{-1})^3(\log\log n)^{2}}{2\log n}-\frac{(1-v^{-1})^{2}(1-\log2\Gamma(\frac{1}{v})+x)\log\log n}{\log n}\nonumber\\
&&+o\big(\frac{\log\log n}{\log n}\big)\Big]\Big\}\frac{e^{-rx}}{(r-1)!}+O(n^{-1})\nonumber\\
&=&\Lambda(x)\Big[\frac{(1-v^{-1})^3(\log\log n)^{2}}{2\log n}-\frac{(1-v^{-1})^{2}(1-\log2\Gamma(\frac{1}{v})+x)\log\log n}{\log n}\Big]\frac{e^{-rx}}{(r-1)!}+o\big(\frac{\log\log n}{\log n}\big),
\end{eqnarray*}
implies
\begin{eqnarray*}
&&\lim_{n\to\infty}(\frac{\log n}{(\log\log n)^2})\big[\P\left(|M_{n,r}|^{p}\le \alpha_{n}^{*}x+\beta_{n}^{*}\right)-\Lambda_{r}(x)\big]=\frac{(1-v^{-1})^{3}e^{-rx}}{2(r-1)!}\Lambda(x)
\end{eqnarray*}
and
\begin{eqnarray*}
&&\lim_{n\to\infty}(\log\log n)\Big\{(\frac{\log n}{(\log\log n)^2})\big[\P\left(|M_{n,r}|^{p}\le \alpha_{n}^{*}x+\beta_{n}^{*}\right)-\Lambda_{r}(x)\big]-\frac{(1-v^{-1})^{3}e^{-rx}}{2(r-1)!}\Lambda(x)\Big\}\\
&&=-(1-v^{-1})^{2}\left(1-\log2\Gamma(\frac{1}{v})+x\right)\frac{e^{-rx}}{(r-1)!}\Lambda(x).
\end{eqnarray*}
The proof is complete.
\qed

\noindent{\bf Proof of Theorem \ref{thm2}.}
(i). For the case of $v\in(0,1)\bigcup(1,+\infty)$ and $p\neq v$, let $z_{n,p}(x)=(c_{n}x+d_{n})^{1/p}$ with $c_{n}$ and $d_{n}$ given by Theorem \ref{thm2}(i).  By using \eqref{eq3.4} we have
\begin{eqnarray}
\label{eq4.5}
 1-\theta_{n,p}(x)=h_{v}(x)e^{x}b_{n}^{-v}+q_{v}(x)e^{x}b_{n}^{-2v}+o(b_{n}^{-2v}),
\end{eqnarray}
where  $\theta_{n,p}(x)=ne^{x}(1-G_{v}(z_{n,p}(x))$, and $h_{v}(x)$ and $q_{v}(x)$ are given by \eqref{eq2.1} and \eqref{eq2.2}, respectively.
It follows from Lemma \ref{lemma3} and \eqref{eq4.5} that
\begin{eqnarray}
\label{add5}
&&\P\big(|M_{n,r}|^p\leq c_{n}x+d_{n}\big) -\Lambda_{r}(x)\nonumber\\
&=&\Lambda(x)\left[1-\frac{1}{2}(1-\theta_{n,p}(x))(r-1-e^{-x})\right](1-\theta_{n,p}(x))\frac{e^{-rx}}{(r-1)!}+O(n^{-1})\nonumber\\
&=&\Lambda(x)\left[h_{v}(x)e^{x}b_{n}^{-v}+q_{v}(x)e^{x}b_{n}^{-2v}+o(b_{n}^{-2v})\right]\nonumber\\
&&\times\Big\{1-\frac{r-1-e^{-x}}{2}\left[h_{v}(x)e^{x}b_{n}^{-v}+q_{v}(x)e^{x}b_{n}^{-2v}+o(b_{n}^{-2v})\right]\Big\}\frac{e^{-rx}}{(r-1)!}+O(n^{-1})\nonumber\\
&=&\Lambda(x)\left[h_{v}(x)b_{n}^{-v}+\left(q_{v}(x)+\frac{1-(r-1)e^{x}}{2}h_{v}^{2}(x)\right)b_{n}^{-2v}\right]\frac{e^{-(r-1)x}}{(r-1)!}+o(b_{n}^{-2v}).
\end{eqnarray}
Hence,  it follows from \eqref{add5}  that
\begin{eqnarray*}
&&\lim_{n\to\infty}b_n^v\Big[\P\left(|M_{n,r}|^p\leq c_nx+d_n\right)-\Lambda_{r}(x)\big]=\Lambda(x)h_v(x)\frac{e^{-(r-1)x}}{(r-1)!}
\end{eqnarray*}
and
   \begin{eqnarray*}
   &&\lim_{n\to\infty}b_n^v\bigg\{b_n^v\Big[\P\left(|M_{n,r}|^p\leq c_nx+d_n\right)-\Lambda_{r}(x)\big]-\Lambda(x)h_v(x)\frac{e^{-(r-1)x}}{(r-1)!}\bigg\}\\
   &&=\left[q_v(x)+\left(1-(r-1)e^{x}\right)\frac{h_v^2(x)}{2}\right]\frac{e^{-(r-1)x}}{(r-1)!}\Lambda(x).
   \end{eqnarray*}

(ii). For the case of $v\in(0,1)\bigcup(1,+\infty)$ and $p= v$, let $z_{n,p}(x)=(c_{n}^{*}x+d_{n}^{*})^{1/v}$ with $c_{n}^{*}$ and $d_{n}^{*}$ given by Theorem \ref{thm2}(ii). With $\theta_{n,v}(x)=ne^{x}(1-G_{v}(z_{n,v}(x))$  and $s_{v}(x)$ and $b_{v}(x)$ given by \eqref{eq2.3} and \eqref{eq2.4}, it follows from \eqref{eq3.5} that
\begin{eqnarray}
\label{eq4.6}
1-\theta_{n,v}(x)=s_{v}(x)e^{x}b_{n}^{-2v}+b_{v}(x)e^{x}b_{n}^{-3v}+o(b_{n}^{-3v}).
   \end{eqnarray}
  It follows from Lemma \ref{lemma3} and \eqref{eq4.6} that
   \begin{eqnarray}
   \label{add6}
   &&\P\big(|M_{n,r}|^{v}\leq c_{n}^{*}x+d_{n}^{*}\big) -\Lambda_{r}(x)\nonumber\\
   &=&\Lambda(x)\left[1-\frac{1}{2}(1-\theta_{n,v}(x))(r-1-e^{-x})\right](1-\theta_{n,v}(x))\frac{e^{-rx}}{(r-1)!}+O(n^{-1})\nonumber\\
   &=&\Lambda(x)\left[s_{v}(x)e^{x}b_{n}^{-2v}+b_{v}(x)e^{x}b_{n}^{-3v}+o(b_{n}^{-3v})\right]\nonumber\\
   &\times& \Big\{1-\frac{r-1-e^{-x}}{2}\left[S_{v}(x)e^{x}b_{n}^{-2v}+B(x)e^{x}b_{n}^{-3v}+o(b_{n}^{-3v})\right]\Big\}\frac{e^{-rx}}{(r-1)!}+O(n^{-1})\nonumber\\
   &=&\Lambda(x)\left[s_{v}(x)b_{n}^{-2v}+b_{v}(x)b_{n}^{-3v}\right]\frac{e^{-(r-1)x}}{(r-1)!}+o(b_{n}^{-3v}),
    \end{eqnarray}
which implies
\begin{eqnarray*}
\lim_{n\to\infty}b_n^{2v}\Big[\P\left(|M_{n,r}|^v\leq c_n^*x+d_n^*\right)-\Lambda_{r}(x)\Big]=\Lambda(x)s_v(x)\frac{e^{-(r-1)x}}{(r-1)!}
\end{eqnarray*}
and
\begin{eqnarray*}
\lim_{n\to\infty}b_n^{v}\bigg\{b_n^{2v}\Big[\P\left(|M_{n,r}|^v\leq c_n^*x+d_n^*\right)-\Lambda_{r}(x)\Big]-\Lambda(x)s_v(x)\frac{e^{-(r-1)x}}{(r-1)!}\bigg\}=\Lambda(x)b_{v}(x)\frac{e^{-(r-1)x}}{(r-1)!}.
\end{eqnarray*}
The proof is complete.\qed


\end{document}